\documentclass[12pt]{article} 
\usepackage{amsmath}
\usepackage{amssymb}
\usepackage{amsthm}
\usepackage[top=1in, bottom=1in, left=1in, right=1in]{geometry}
\usepackage{xcolor}
\usepackage{hyperref}

\newtheorem{theorem}{Theorem}[section]
\newtheorem{lemma}[theorem]{Lemma}
\newtheorem{prop}[theorem]{Proposition}
\newtheorem{coro}[theorem]{Corollary}

\newtheorem{remark}[theorem]{Remark}

\newcommand{\ddbar}{\partial\bar\partial}

\renewcommand{\o}{\omega}

\title{Extremizers of the $J$ functional with respect to the $d_1$ metric}
\author{Sam Bachhuber, Aaron Benda, Benjamin Christophel, Tam\'as Darvas}
\date{\small{\emph{Dedicated to L\'aszl\'o Lempert on the occasion of his 70th birthday.}}}
\begin{document}
\maketitle
\begin{abstract} In previous work, Darvas--George--Smith obtained  inequalities between the large scale asymptotic of the $J$ functional with respect to the $d_1$ metric on the space of toric K\"ahler metrics/rays. 
In this work we prove sharpness of these inequalities on all toric K\"ahler manifolds, and study the extremizing potentials/rays. On general K\"ahler manifolds we show that existence of radial extremizers is equivalent with the existence of plurisupported currents, as introduced and studied by McCleerey.

\end{abstract}

\section{Introduction and main results}

 Let $(X,\omega)$ be a K\"ahler manifold of dimension $n$. We consider the space of K\"ahler metrics that are cohomologous to $\omega$:
$$\mathcal H := \{\tilde \omega \textup{ K\"ahler on } X \textup{ and } [\tilde \omega]_{dR} = [\omega]_{dR}\}$$

By the $\ddbar$-lemma, for all $\tilde \omega \in \mathcal H$ there exists $u \in C^\infty(X)$, unique up to a constant, such that $\tilde \omega = \omega_u:= \omega + i\ddbar u.$ Consequently,  instead of looking at $\mathcal H$ directly, one can work with the space of \emph{K\"ahler potentials} instead:
$$\mathcal H_\omega :=\{ u \in C^\infty(X) \textup{ s.t. } \omega + i\ddbar u >0 \}.$$
By $\textup{PSH}(X,\omega)$ we denote the space of quasi-plurisubharmonic functions $u$ satisfying $\omega_u:= \omega + i\ddbar u \geq 0$, in the sense of currents. Clearly, $\mathcal H_\omega \subset \textup{PSH}(X,\omega)$, hence all K\"ahler potentials are $\o$-plurisubharmonic ($\omega$-psh). For a comprehensive treatment of $\omega$-psh functions we refer to the recent book \cite[Chapter 8]{GZ16}, for a quick introduction see  \cite[Appendix A.1]{Da19}, \cite[Section 2]{Bl13}.

Let $I:\mathcal H_\o \to \Bbb R$ be the \emph{Monge--Amp\`ere energy}, one of the most basic functionals of K\"ahler geometry:
\begin{equation}\label{eq: I_def}
I(u):= \frac{1}{(n+1)V}\sum_{j=0}^n \int_X u \o^j \wedge \o_{u}^{n-j},
\end{equation}
where $V = \int_X \omega^n$ is the volume of $(X,\omega)$. For a  detailed analysis of the $I$ functional we refer to \cite[page 111]{Bl13} and \cite[Section 3.7]{Da19}. Closely related, the $J$ functional $J: \mathcal H_\omega \to \Bbb R$ is defined as follows:
$$J(u) = \frac{1}{V}\int_X u \o^n - I(u).$$
Using Stokes' theorem one can prove that $J(u) \geq 0$, and in many ways $J$ acts as a norm-like expression on $\mathcal H_\o$. This aspect will be revisited in this work.

By definition, the space of K\"ahler potentials $\mathcal H_\o$ is a convex open subset of $C^\infty(X)$, hence it is a trivial ``Fr\'echet manifold". Motivated by questions in stability, one can introduce on $\mathcal H_\o$ an $L^1$ type Finsler metric \cite{Da15}. If $u \in \mathcal H_\o$ and $\xi \in T_u \mathcal H_\o \simeq C^\infty(X)$, then the $L^1$-length of $\xi$ is given by the following expression:
\begin{equation}\label{eq: Lp_metric_def}
\| \xi\|_{u} = \frac{1}{V}\int_X |\xi| \o_u^n.
\end{equation}
The analogous $L^2$ type metric was introduced by Mabuchi \cite{Ma87} (independently discovered by Semmes \cite{Se92} and Donaldson \cite{Do99}). This Riemannian metric was studied in detail by Chen \cite{Ch00}, and for more historical details we refer to \cite[Chapter 3]{Da19}.

To the Finsler metric in \eqref{eq: Lp_metric_def} one associates a path length pseudo-metric $d_1(\cdot,\cdot)$ on $\mathcal H_\omega$. As proved in \cite{Da15}, $d_1$ is actually a metric and $(\mathcal H_\omega,d_1)$ is a geodesic metric space, whose abstract completion can be identified with $(
\mathcal E^1,d_1)$, where $\mathcal E^1 \subset \textup{PSH}(X,\omega)$ is a space of potentials introduced by Guedj--Zeriahi \cite{GZ07}, inspired by work of Cegrell \cite{Ce98} in the local case. 
\paragraph{Extremizing potentials in toric case and convex analysis.}By \cite[Remark 6.3]{Da15} it was known that  $d_1$ and $J$ are asymptotically equivalent on $\mathcal H_\omega$. In \cite{DGS21} the authors set out to sharpen the  constants in the asymptotic comparison of $J$ and $d_1$, and proved the following global inequalities on toric K\"ahler manifolds:

\begin{theorem}[\cite{DGS21}]\label{thm: main_thm} Let $(X,\omega)$ be a toric K\"ahler manifold. Then there exists a constant $D>0$ such that 
\begin{equation}\label{eq: main_thm}
\frac{2}{n+1}\cdot{\left(\frac{n}{n+1}\right)}^n J(u)  - D \leq d_1(0,u)\leq 2 J(u) + D, \ \ u \in \mathcal H^T_\omega \cap I^{-1}(0).
\end{equation}
In case $X = \Bbb CP^n$ and $\omega$ is the Fubini-Study metric the constants multiplying $J(u)$ are sharp.
\end{theorem}

The motivation for obtaining sharpness in \eqref{eq: main_thm} stems from criteria of existence for canonical K\"ahler metrics. It is expected that uniform K-stability of $(X,\omega)$ is related to energy properness with respect to the $d_1$ metric \cite{BBJ}. On the other hand, in concrete situations one can often show energy properness with respect to the $J$ functional, also implying existence of canonical metrics (e.g. \cite{SZ20}). Naturally, there is an interest in converting such properness criteria into K-stability criteria, and vice versa, as optimally as possible.

In \cite[page 3]{DGS21} the authors proposed to study the sharpness of the above inequality in case of arbitrary toric K\"ahler manifolds, not just projective space. This is what we address in this work, and we show that  \eqref{eq: main_thm} is sharp for any toric K\"ahler manifold.

As pointed out in \cite[Section 2]{DGS21}, using Legendre transforms, the pair of inequalities in \eqref{eq: main_thm} is equivalent to a pair of integral inequalities involving convex functions of a convex domain, as expressed in \eqref{eq: main_thm_convex} below. We show that this pair of inequalities is sharp on any convex domain, proving the sharpness of \eqref{eq: main_thm} on any toric K\"ahler manifold:

\begin{theorem}\label{thm: main_thm_convex_1} Let $P \subset \Bbb R^n$ be a bounded open convex set. \\
\noindent (i)\cite{DGS21} for any convex $\phi \in L^1(P)$, satisfying $\int_P \phi= 0$, the following inequality holds:  
\begin{equation}\label{eq: main_thm_convex} 
 -\frac{2}{n+1}\cdot{\left(\frac{n}{n+1}\right)}^n\inf_P \phi \leq \frac{1}{\mu(P)}\int_P |\phi| d\mu\leq -2\inf_P \phi,
 \end{equation}
\noindent (ii) there exists a convex function $\phi \in L^1(P)$ such that $\int_P \phi = 0$ and 
\begin{equation}\label{eq: main_thm_convex_extr} 
 -\frac{2}{n+1}\cdot{\left(\frac{n}{n+1}\right)}^n\inf_P \phi = \frac{1}{\mu(P)}\int_P |\phi| d\mu.
\end{equation}
\end{theorem}

Potentials $\phi$ satisfying \eqref{eq: main_thm_convex_extr}  will be referred to as \emph{extremizers} in this work. When $P$ is the unit simplex (corresponding to projective space), the existence of extremizers was proved in \cite[Proposition 4.2]{DGS21}.

Regarding the first inequality in \eqref{eq: main_thm_convex}, we don't just prove the existence of extremizers, but we end up characterizing them explicitly, with this answering another question posed in \cite[page 3]{DGS21} (See Theorem \ref{thm: extremizer_char}). What is more, in Section 3.1 we characterize domains $P$ admitting affine extremizers. Such domains $P$ are always very special: they are the intersection of a cone with a half-space.

Regarding the second inequality in \eqref{eq: main_thm_convex}, it is easy to see that it is always sharp, but without any extremizing functions (See Remark \ref{rem: not_interesting}).

\paragraph{Extremizing rays and plurisupported currents.}
In \cite{DGS21} it was conjectured that Theorem \ref{thm: main_thm} holds in case of general K\"ahler manifolds as well.  To provide evidence for this, the authors obtained the radial version of the expected result, as we now recall, along with related terminology from \cite{DL18}. 

Let $\{u_t\} \in \mathcal R^1$ be a $d_1$-geodesic ray $[0,\infty) \ni t \to u_t \in \mathcal E^1$ of $(\mathcal E^1,d_1)$, emanating from $u_0 = 0 \in \mathcal H_\omega$, and normalized by $I(u_t) = 0, \ t \geq 0$. Since the $J$ functional is convex along geodesics, one can define its slope along geodesic rays (informally called the \emph{radial $J$ functional}):
\begin{equation}\label{eq: J_radial_def}
\{u_t\} := \lim_{t \to \infty} \frac{J(u_t)}{t}.
\end{equation}
The $L^1$ speed of a ray $\{u_t\}_t$ is simply the quantity $d_1(0,u_1)$.  We have the following inequality between the $L^1$ speed and the radial $J$ functional:

\begin{theorem}\cite{DGS21}\label{thm: radial_main_ineq} Suppose that $(X,\omega)$ is a compact K\"ahler manifold and $\{u_t\}_t \in \mathcal R^1$. Then the following inequality holds:
\begin{equation}\label{eq: radial_main_ineq}
\frac{2}{n+1}\cdot{\left(\frac{n}{n+1}\right)}^n J\{u_t\} \leq d_1(0,u_1)\leq 2 J\{u_t\}.
\end{equation}
The inequality is sharp in case of $(\Bbb C\Bbb P^n, \omega_{FS}).$
\end{theorem}

Since rays are constant speed, notice that the middle term in the above inequality could have been replaced by the limit $\lim_{t \to \infty} \frac{d_1(0,u_t)}{t}$. In particular, \eqref{eq: radial_main_ineq} would instantly follow from the conjectured inequality \eqref{eq: main_thm} for general K\"ahler manifolds.

The sharpness of the second inequality in \eqref{eq: radial_main_ineq} is not difficult to argue on any K\"ahler manifold, using examples similar to Remark \ref{rem: not_interesting}. As in Theorem \ref{thm: main_thm_convex_1}(ii), we are interested in the sharpness of the first inequality of \eqref{eq: radial_main_ineq} on arbitrary K\"ahler manifolds, and finding all extremizing rays.

It turns out that this question is tied to the existence of plurisupported potentials in $\textup{PSH}(X,\omega)$, as introduced and studied in \cite{MC21}. Recall that $v \in \textup{PSH}(X,\omega)$ is \emph{plurisupported} if $\int_X \omega_ v \wedge \omega^{n-1} =0$, i.e., the current $\omega_v$ is supported on a pluripolar set. 

\begin{theorem}\label{thm: main_extremal_ray} Suppose that $(X,\omega)$ is a compact K\"ahler manifold. Then there exists $\{u_t\}_t \in \mathcal R^1$ such that $J\{u_t\}/d_1(0,u_1) = \frac{(n+1)^{n+1}}{2 n^n}$ if and only if there exists $v \in \textup{PSH}(X,\omega)$ plurisupported.
\end{theorem}

The construction of the extremizings rays $\{u_t\}_t \in \mathcal R^1$ in terms of plurisupported currents $v \in \textup{PSH}(X,\omega)$  is quite explicit, and is elaborated in Theorem \ref{thm: extr_ray_char}. What is more, it follows from Theorem \ref{thm: extr_ray_char} that all such extremizing rays come from this same construction.

Existence of plurisupported potentials is a deep question related to a number of important conjectures in the field. For details, we refer to \cite{MC21}. However if $(X,\omega)$ is a projective K\"ahler manifold, then there exists $v \in \textup{PSH}(X,\omega)$ such that $\omega_v = [D]$, for some divisor $D \subset X$. It is clear that such $v$ is plurisupported. Hence we get the following corollary:

\begin{coro} Suppose that $(X,\omega)$ is a projective K\"ahler manifold. Then there exists $\{u_t\}_t \in \mathcal R^1$ such that $J\{u_t\}/d_1(0,u_1) = \frac{(n+1)^{n+1}}{2 n^n}$. In particular, the first inequality of \eqref{eq: radial_main_ineq} is sharp.
\end{coro}

\paragraph{Acknowledgments.} This REU project was carried out during the Spring and Summer term of 2021, partially funded by the NSF grant DMS-1846942(CAREER) and an Alfred P. Sloan Fellowship. We thank the anonymous referee for recommendations that improved the presentation. 

After the publication of the paper in Analysis Mathematica, we learned from Kewei Zhang that the constants of Theorem 1.2 and Theorem 1.3 already appeared in the non-Archimedean context in \cite[Theorem 7.9]{BHJ17}, without discussing their optimality. Though the results don't overlap, the techniques of the arguments do. It is remarkable how the same philosphy is applicable in the totally different non-Archimedean and toric cases. Both this paper and \cite{DGS21} would have been presented differently had we been aware of \cite[Theorem 7.9]{BHJ17}.

\section{Preliminaries}

\subsection{Analysis on toric K\"ahler manifolds}

In this short section we point out how the inequalities \eqref{eq: main_thm} and \eqref{eq: main_thm_convex} are equivalent. The material here closely follows \cite[Section 3]{DGS21}, itself inspired from \cite[Section 6]{DG16} and \cite{ZZ08}. We say that $(X,\omega)$ is a toric K\"ahler manifold of dimension $n$ if we can embed $(\Bbb C^*)^n$ into $X$ so that the complement of $(\Bbb C^*)^n$ inside $X$ is a Zariski closed set. Moreover, we ask that the trivial action of $\Bbb T^n := (S^1)^n$ on $(\Bbb C^*)^n$ extends to $X$, and the K\"ahler form $\omega$ is $\Bbb T^n$-invariant.

Using the fact that $\omega$ is $\Bbb T^n$-invariant we get that
\begin{equation}\label{eq: omega_triv}
\omega = i\ddbar (\psi_0 \circ L) \ \textup{ on } \ (\Bbb C^*)^n,
\end{equation}
where $\psi_0 \in C^\infty(\Bbb R^n)$ and $L(z_1,z_2,\ldots,z_n) = (\log|z_1|,\log|z_2|,\ldots,\log|z_n|) \in \Bbb R^n.$ Since $\psi_0 \circ L$ is psh on $(\Bbb C^*)^n$, the potential $\psi_0$ has to be strictly convex on $\Bbb R^n$, and we may choose $\psi_0(0)=0$.

By $\mathcal H^{T}$ we will denote the metrics $\omega' \in \mathcal H$ that are torus invariant, i.e., $(S^1)^n$ acts by isometries on $\omega'$. $\mathcal H^{T}_\omega$ will denote the corresponding space of toric potentials.

For  $u \in \mathcal H^{T}_\omega$ we can introduce the following potential 
$$\psi_u := \psi_0 + u \circ E,$$
where $E(x)=E(x_1,x_2,\ldots,x_n) := (e^{x_1},e^{x_2},\ldots,e^{x_n}), \ x \in \Bbb R^n.$  The point is that $\omega_u = i\ddbar \psi_u \circ L$ on $(\Bbb C^*)^n$.

\paragraph{The Legendre transform.} Given $\omega_u \in \mathcal H_\o^T$, it follows from a result of  Atiyah--Guillemin-- (see \cite[Chapter 27]{CS08}) that  the ``moment
map" $\nabla \psi_u: \Bbb R^n \to \Bbb R^n$ is one-to-one and sends $\Bbb R^n$ to $P := \textup{Im }\nabla \psi_u$, which is a convex bounded polytope, independent of $u$. 

Though we will not use it, by a theorem of Delzant, $P$ satisfies a number of additional properties ({simple}, i.e., there are $n$ edges meeting at each vertex; {rational}, i.e., the edges meeting at the vertex $p$ are rational in the sense that each edge is of the form $p + t u_i$, $0 \leq t < \infty$, where $u_i \in \Bbb Z^n$; {smooth}, i.e., these $u_1,\ldots,u_n$ can be chosen to be a basis of $\Bbb Z^n$). What is more, such Delzant polytopes $P$ determine toric K\"ahler structures $(X,\omega)$ uniquely (see \cite{Ab00} and \cite[Chapter 28]{CS08}).

Since $\psi_u: \Bbb R^n \to \Bbb R$ is convex, we can take its Legendre transform and obtain a dual convex function $\phi_u$,  with possible values equal to $+\infty$:
$$\phi_u(s) = \psi_u^*(s):= \sup_{x \in \Bbb R^n}\big(\langle s, x\rangle - \psi_u(x) \big), \ \ s \in \Bbb R^n.$$

Since $\textup{Im }\nabla \psi_u = P$, it follows that $\phi_u(s)$ is finite if and only if $s \in P$. Also,  by the involutive property, for all $x \in \Bbb R^n$ we will have
$$\phi^*_u = \psi^{**}_u = \psi_u,$$
$$\phi_u(\nabla \psi_u(x)) = \langle x,\nabla \psi_u(x) \rangle - \psi_u(x) \ \ \textup{ and } \ \  \nabla \psi_u(x) = s \Leftrightarrow \nabla \phi_u(s) = x.$$

Summarizing, the Legendre transform $u \to \psi_u^*=\phi_u$ gives a one-to-one correspondence between elements of $\mathcal H_\omega^T$ and the class $\mathcal C(P)$:
$$\mathcal C(P) := \{f:P \to \Bbb R \textup{ is convex and } f - \phi_0 \in C^\infty(\overline{P})\},$$
where $\phi_0 = \psi_0^*$. In addition, as pointed out in \cite[Proposition 4.5]{Gu14},  there is a one-to-one correspondence between $(S^1)^n$-invariant elements of $\textup{PSH}(X,\omega)\cap L^\infty$ and convex functions $f:P \to \Bbb R$ for which $f - \phi_0$ is only bounded on $P$. In particular, $\phi_0 \in L^\infty(P)$.

\paragraph{The $L^1$ geometry of toric metrics.} We briefly describe the $L^1$ geometry of $\mathcal H_\omega$ restricted to $\mathcal H_u^T$.
Let $[0,1]\ni t \to u_t \in \mathcal H_\omega^T$ be a smooth curve  connecting $u_0,u_1\in \mathcal H_\omega^T$. Taking the Legendre transform of the potentials $\psi_{u_t}$ we obtain
\begin{equation}\label{eq: phi_u_t_def}
\phi_{u_t}(s):=\sup_{x \in \Bbb R^n} \left\{ \langle x,s \rangle -\psi_{u_t}(x) \right\}
=\langle x_t,s \rangle -\psi_{u_t}(x_t), \ s \in \Bbb R^n,
\end{equation}
where $x_t=x_t(s)$ is such that $\nabla \psi_{u_t}(x_t)=s$. 

After elementary but tedious calculations (see \cite[Section 3]{DGS21}) we conclude that
\begin{equation}\label{eq: L^1_isom}
\int_{X} |\dot u_t| \omega_{u_t}^n = \pi^n n! \int_{\Bbb R^n} |\dot \psi_{u_t}| MA_{\Bbb R}(\psi_{u_t})= \pi^n n! \int_P |\dot{\phi}_{u_t}(s)| d \mu(s),
\end{equation}
Hence the Legendre transform sends the $L^1$ geometry of $\mathcal H_\omega^T$ to the flat $L^1$ geometry of convex functions on $P$. In the particular case when $u_t := t$, we obtain the following useful formula about the K\"ahler volume:
\begin{equation}\label{eq: volume_formula}
\int_X \omega^n = V = \pi^n n! \int_P d\mu.
\end{equation}

Regarding the underlying path length metrics,  \eqref{eq: L^1_isom} has the following important consequence:
\begin{theorem}\label{thm: d_1_toric}Suppose $u_0,u_1 \in \mathcal H_\omega^{T}$. Then
\begin{equation}\label{eq: d_1_toric}
d_1(u_0,u_1)=\frac{1}{ \mu(P)} \int_P |\phi_{u_0}(s)-\phi_{u_1}(s)| d \mu(s),
\end{equation}
where $\mu(P)$ is the Lebesgue measure of $P$.
\end{theorem}

As a consequence of the above theorem, one concludes that the $d_1$-completion of $\mathcal H^T_\omega$ can be identified with the space of convex functions on $P$, that are also integrable on $P$.

\paragraph{The $I$ and $J$ functionals of toric metrics.} In this paragraph we analyze the $I$ and $J$ functionals in terms of the Legendre transform.

The $I$ energy is essentially the Lebesgue integral, after applying the Legendre transform. Indeed, let $[0,1] \ni t \to v_t \in \mathcal H_\omega^T$ be any smooth curve connecting $v_0 =0$ and $v_1 = u$. After elementary, but tedious calculations (see \cite[Section 3]{DGS21}), one obtains the following formula:
\begin{flalign}\label{eq: Aubin_Yau_toric}
I(u) &= I(u)-I(0) =\frac{-1}{\mu(P)}  \int_0^1 \int_{P} \dot \phi_{v_t}(s) d \mu(s) dt=\frac{-1}{\mu(P)} \int_P (\phi_u(s) - \phi_0(s))d \mu(s).
\end{flalign}

Since $\phi_0 \in L^\infty(P)$ \cite[Proposition 4.5]{Gu14}, we conclude that there exists $C = C(X,\omega)>0$ such that 
\begin{equation}\label{eq: I_magn_est}
 \frac{-1}{\mu(P)} \int_P \phi_u d\mu - C \leq I(u) \leq \frac{-1}{\mu(P)} \int_P \phi_u d\mu + C, \ \ u \in \mathcal H_\omega^T.
\end{equation}

Regarding the $J$ energy, we can express its magnitude using the Legendre transform. This will be sufficient for our analysis (see \cite[Proposition 3.2]{DGS21}):
\begin{prop}\label{prop: J_toric} There exists $C:=C(X,\omega)>0$ such that for all $u \in \mathcal H_\o^T$ with $I(u)=0$ we have
$$-\inf_P \phi_u - C \leq J(u) = \frac{1}{V}\int_X u \omega^n \leq -\inf_P \phi_u + C.$$
\end{prop}

Putting together Theorem \ref{thm: d_1_toric}, \eqref{eq: I_magn_est}, and Proposition \ref{prop: J_toric}, we obtain the equivalence of the inequalities \eqref{eq: main_thm} and \eqref{eq: main_thm_convex}, as desired at the beginning of this subsection.

\subsection{The Ross--Witt Nystr\"om correspondence}

In this subsection we recall the correspondence between geodesic rays and  test curves, as established in \cite[Section 3.1]{DX20} and \cite[Section 4]{DDL3}, building on \cite{RWN14}.

A \emph{sublinear subgeodesic ray} is a subgeodesic ray $(0,+\infty) \ni t \mapsto u_t \in \textup{PSH}(X,\omega)$ (notation $\{u_t\}_{t >0}$) such that $u_t \to_{L^1} u_0: = 0$ as $t \to 0$, and there exists $C\in \mathbb{R}$ such that $u_t(x) \leq C t$ for all $t\geq 0$, $x \in X$.

A \emph{psh geodesic ray} is a sublinear subgeodesic ray that additionally satisfies the following maximality property: for any $0 < a < b$, the subgeodesic $(0,1) \ni t \mapsto v^{a,b}_t:=u_{a(1-t) + bt} \in \textup{PSH}(X,\omega)$ can be recovered in the following manner:
\begin{equation}\label{eq: vabt_eq}
v^{a,b}_t:=\sup_{h \in \mathcal{S}}{h_t}\,,\quad t \in [0,1]\,,
\end{equation}
where $\mathcal{S}$ is the set of subgeodesics $(0, 1) \ni  t \to h_t \in \textup{PSH}(X,\omega)$ such that
\[
\lim_{t \searrow 0}h_t\leq  u_a\,,\quad \lim_{t \nearrow 1}h_t\leq u_b\,.
\]

We note the following properties of the map $v \mapsto \sup_X v$ along rays, established in \cite[Lemma 3.1]{DX20}:
\begin{lemma}\label{lem: suplinear} For any psh geodesic ray $\{u_t\}_t$, the map $t \mapsto \sup_X u_t$ is linear. For sublinear subgeodesics, the map $t \mapsto \sup_X u_t$ is convex.
\end{lemma}

The statement for subgeodesics follows from $t$-convexity.  To argue the statement for psh geodesic rays, one can use \cite[Theorem~1]{Da17} together with approximation by bounded geodesics, and the continuity of $u \mapsto \sup_X u$ in the weak $L^1$-topology of $\textup{PSH}(X,\omega)$.\medskip

A map $\mathbb{R}\ni \tau \mapsto \psi_\tau\in \textup{PSH}(X, \omega)$ is a \emph{psh test curve}, denoted $\{\psi_\tau\}_{\tau \in \mathbb R}$, if \vspace{0.15cm}\\
(i) $\tau\mapsto \psi_\tau(x)$ is concave, decreasing and usc for any $x\in X$. \vspace{0.15cm}\\
(ii) $\psi_\tau\equiv -\infty$ for all $\tau$ big enough, and $\psi_\tau$ increases a.e. to $0$ as $\tau \to -\infty$.\medskip

\noindent Condition (ii) allows for the introduction of the following constant:
\begin{equation}\label{eq: psi+_tau_def}
\tau_\psi^+ := \inf\{\tau \in \mathbb{R} : \psi_\tau \equiv -\infty\}\,.
\end{equation}

\begin{remark}\label{rem: rays_pot} We adopt the following notational convention: psh test curves will always be parametrized by $\tau$, whereas rays will be parametrized by $t$. Hence $\{\psi_t\}_t$ will always refer to some type of ray, whereas  $\{\phi_\tau\}_\tau$ will refer to some kind of test curve. As recalled below, rays and test curves are dual to each other, so one should think of the  parameters $t$ and $\tau$ as duals to each other as well.
\end{remark}

Given $v \in \textup{PSH}(X,\omega)$ we consider the following notion of envelope:
\begin{equation}\label{eq: P[v]_def}
P[v] = \textup{usc}\big( \sup\{w \in \textup{PSH}(X,\omega), \ w \leq 0, \ w \leq v + C \textup{ for some } C >0\} \big)
\end{equation}

A  psh test curve $\{\psi_\tau \}_\tau$  can have the following properties: \vspace{0.1cm}\\
(i) $\{\psi_{\tau}\}_\tau$ is \emph{maximal} if $P[\psi_\tau] =\psi_\tau, \ \tau \in \mathbb{R}$.\vspace{0.1cm}\\
(ii) $\{\psi_{\tau}\}_\tau$ has \emph{finite energy} if
    \begin{equation}\label{eq: fetestcurve_def}
        \int_{-\infty}^{\tau^+_\psi} \left( \int_X \omega_{\psi_\tau}^n-\int_X \omega^n \right) \,\mathrm{d}\tau >-\infty\,.
    \end{equation}
(iii) We say $\{\psi_{\tau}\}_\tau$ is \emph{bounded} if $\psi_\tau = 0$ for $\tau$ small enough. In this case, we introduce the following constant, complementing \eqref{eq: psi+_tau_def}:
 \begin{equation}\label{eq: psi-_tau_def}
\tau_\psi^- := \sup\left\{\,\tau \in \mathbb{R} : \psi_\tau \equiv 0\,\right\}\,.
\end{equation}

Above we followed the convention $P[-\infty]=-\infty$, and we note that bounded test curves are clearly of finite energy.

 We recall the \emph{Legendre transform}, helping establish the duality between various types of maximal test curves and geodesic rays. Given a convex function $f:[0, +\infty)\rightarrow \mathbb{R}$, its Legendre transform is defined as 
 \begin{equation}
 \hat f(\tau):= \inf_{t \geq 0} (f(t)-t\tau)=\inf_{t > 0} (f(t)-t\tau)\,,\quad \tau\in \mathbb{R}\,.
 \end{equation}
The \emph{(inverse) Legendre transform} of a decreasing concave function $g:\mathbb{R}\rightarrow \mathbb{R}\cup \{-\infty\}$ is
\begin{equation}\label{eq: inverse_Lag_tran_def}
\check{g}(t):=\sup_{\tau \in \mathbb{R}} (g(\tau)+t\tau)\,, \quad t \geq 0\,.
\end{equation}
There is a sign difference in our choice of Legendre transform compared to the convex analysis literature, however this choice is more suitable for us.

Given a sublinear subgeodesic ray $\{\phi_t\}_t$ (psh test curve $\{\psi_\tau\}_\tau$), we can associate its (inverse) Legendre transform at $x \in X$ as 
\begin{equation}\label{eq: Leg_transf_def_ray_test_curve}
    \begin{aligned}
\hat \phi_\tau(x) := \inf_{t>0}(\phi_t(x) - t\tau)\,,& \quad \tau \in \mathbb R\,,\\
\check \psi_t(x) := \sup_{\tau \in \mathbb R}(\psi_\tau(x) + t\tau)\,,& \quad t> 0\,.
\end{aligned}
\end{equation}

The Ross--Witt Nystrom correspondence describes the duality between various types of rays and maximal test curves, proved in \cite[Theorem 3.7]{DX20}:
\begin{theorem}\label{thm: max_test_curve_ray_duality}
The Legendre transform $\{\psi_\tau\}_\tau \mapsto \{\check \psi_t\}_t$ gives a bijective map with inverse $\{\phi_t\}_t \mapsto \{\hat \phi_\tau\}_\tau$ between:\vspace{0.1cm}\\
(i) psh test curves and sublinear subgeodesic rays,\vspace{0.1cm}\\
(ii) maximal psh test curves and psh geodesic rays,\vspace{0.1cm}\\
(ii)\cite{RWN14} maximal bounded test curves and bounded geodesic rays. In this case, we additionally have
        \[
        \tau_\psi^- t \leq {\check \psi_t}\leq \tau_\psi^+t\,, \quad t \geq 0\,.
        \]
(iv) maximal finite energy test curves and finite energy geodesic rays. In this case, we additionally have
        \begin{equation}\label{eq: I_RWN_form}
            I\{\psi_t\}:=\frac{I(\check \psi_t)}{t}=\frac{1}{V}\int_{-\infty}^{\tau^+_\psi} \left(\int_X \omega_{\psi_\tau}^n-\int_X \omega^n \right) \,\mathrm{d}\tau+\tau_{\psi}^+,\, t >0.
        \end{equation}
Here $I$ is the Monge--Amp\`ere functional defined in \eqref{eq: I_def}.
\end{theorem}

\section{Extremizing functions on convex bodies}

Let $P \subset \Bbb R^n$ be a bounded open convex set. Due to convexity of $P$ we have $\textup{int}\overline{P} = P$.

For any $f: P \to \Bbb R$ convex, it is well known that the lsc extension of $f:\overline{P} \to (-\infty,\infty]$ is also convex, possibly taking up the value $\infty$. Since there will be no chance for confusion, we will not distinguish $f$ from its lsc extension. Since $\overline{P}$ is compact, the infimum of $f$ will be always realized on this bigger set, and this will play a key role in our discussion below. 

Before we discuss sharpness and extremizing potentials for Theorem \ref{thm: main_thm_convex_1}(i), let us recall the proof of  this result from \cite{DGS21}:

 \begin{theorem}\cite{DGS21}\label{thm: thm_convex} Let $P \subset \Bbb R^n$ be a bounded open convex set. Then for $\phi \in L^1(P)$ convex and satisfying $\int_P \phi d\mu = 0$, the following inequalities hold:  
\begin{equation}\label{eq: thm_convex} 
 -\frac{2}{n+1}\cdot{\left(\frac{n}{n+1}\right)}^n\inf_P \phi \leq \frac{1}{\mu(P)}\int_P |\phi| d\mu \leq -2\inf_P \phi,
 \end{equation}
where the integration is in terms of the Lebesgue measure.
\end{theorem}

\begin{proof} First we argue the second inequality. Let 
\begin{equation}P^\phi_- := \{x \in P : \phi<  0\}  \ 
\  \textup{  and } \  \ P^\phi_+ := \{x \in P : \phi \geq 0\} 
\end{equation}
 Since $\int_P\phi=\int_{P^\phi_-}\phi+\int_{P^\phi_+}\phi=0,$ we have
\[\int_P |\phi| = -\int_{P^\phi_-} \phi + \int_{P^\phi_+} \phi = -2\int_{P^\phi_-} \phi.\]
Furthermore, $\int_{P^\phi_-} \phi \geq \mu\big(P^\phi_-\big)\inf_{P^\phi_-} \phi \geq \mu(P) \inf_P \phi$, and the second estimate follows.

We argue the first estimate. For all $a \in \mathbb{R}$, let 
$$P^\phi_a := \{x \in P : \phi(x) <  a\}.$$ 
After re-scaling $\phi$, we can assume without loss of generality that $\inf_P \phi = -1$.  We first claim that 
\begin{equation}\label{eq: claim_est}
\int_P |\phi| \geq -\frac{2}{n+1}\mu\big(P^\phi_-\big).
\end{equation}
Let $y \in \overline {P}$ such that $\phi(y) = \inf_P \phi = -1$.

By convexity of $\phi: \overline{P} \to (-\infty,\infty]$, if $-1< a < b$, we have
\begin{equation} \label{eq: P_a_P_b_est_0}
\frac{b-a}{b+1} y + \frac{a+1}{b+1}P^\phi_b \subset P^\phi_a,  \ \ \textup{ for any } -1 < a < b.
\end{equation}
Taking measures of both sides we arrive at 
\begin{equation} \label{eq: P_a_P_b_est}
\mu\left(P^\phi_b\right) \leq {\left(\frac{b+1}{a+1}\right)}^n\mu\left(P^\phi_a\right), \  \textup{ for any } -1 < a < b.
\end{equation}

Since $\int_P \phi = \int_{P^\phi_+} \phi + \int_{P^\phi_-} \phi = 0$, we have $\int_P |\phi| = 2\int_{P^\phi_-} |\phi|$, allowing to bound $\int_{P} |\phi|$ and proving the claim:
\begin{align}\label{eq: homot_ineq_0}
    \int_P |\phi| &= 2 \int_{P_-} |\phi| = 2\int_{-1}^0 \mu\left(P^\phi_x\right)dx \geq 2\int_{-1}^0 {(1+x)}^n\mu\left(P^\phi_-\right)dx = \frac{2}{n+1}\mu(P^\phi_-),
\end{align}
where we used that $\int f d \mu = \int_0^{+\infty} \mu\{f \geq t\} dt=\int_0^{+\infty} \mu\{f > t\} dt$ for any non-negative $\mu$-measurable $f$, estimate \eqref{eq: P_a_P_b_est} for $-1 < x < 0$, and  $ P^\phi_0 = P^\phi_-$. 

Next we estimate $\int_{P_+}\phi$, applying \eqref{eq: P_a_P_b_est} for $-1 < 0 < x$:
\begin{align} \label{eq: 1/n_ineq}
    \int_{P^\phi_+} \phi &=\int_0^\infty \mu\left(P \setminus P^\phi_x\right)dx\geq \int_0^\frac{1}{n} \mu\left(P \setminus P^\phi_x\right)dx =\int_0^\frac{1}{n} \mu\left(P\right) - \mu\left(P^\phi_x\right)dx\\
    &\geq \int_0^\frac{1}{n} \mu\left(P\right) - {(1+x)}^n\mu\left(P^\phi_-\right)dx \label{eq: homot_ineq}\\
    &= \frac{1}{n}\mu(P) - \frac{{(\frac{1}{n}+1)^{n+1}}}{n+1}\mu(P^\phi_-) + \frac{1}{n+1}\mu(P^\phi_-) \nonumber
\end{align}
We then let $A>0$ such that $\frac{1}{2}\int_{P}|\phi| =\int_{P^\phi_+}|\phi| = \int_{P^\phi_-}|\phi| = A\mu(P^\phi_-)$.  Collecting terms, we arrive at
\[A\mu(P^\phi_-) = \int_{P^\phi_+} \phi\geq \frac{1}{n}\mu(P) - \frac{1}{n}\cdot{\left(\frac{n+1}{n}\right)}^{n}\mu(P^\phi_-) + \frac{1}{n+1}\mu(P^\phi_-),\]
implying
\[\int_P |\phi|=2A\mu(P^\phi_-) \geq \frac{2A}{nA + {\left(\frac{n+1}{n}\right)}^n - \frac{n}{n+1}}\mu(P).\]
The right-hand side is an increasing function of $A$ and by \eqref{eq: claim_est} we know $A \geq \frac{1}{n+1}$.  This means the right hand side is minimized at this value, so
\begin{equation}\label{eq: last_eq}
\int_P |\phi| \geq \frac{\frac{2}{n+1}}{{\left(\frac{n+1}{n}\right)}^n}\mu(P)
= \frac{2}{(n+1)}\cdot{\left(\frac{n}{n+1}\right)}^n\mu(P).
\end{equation}
\end{proof}

\begin{remark}\label{rem: not_interesting} We briefly address the sharpness of the second inequality in \eqref{eq: thm_convex}. For any $P \subset \Bbb R^n$ and $\varepsilon>0$ one can find $\psi_\varepsilon: P \to \Bbb R^+$ convex such that $\mu(\{\psi = 0\})/\mu(P) \geq 1-\varepsilon$ and $\int_X\psi d\mu > 0$. Let $\phi_\varepsilon:= \psi - \int_P \psi$.

Then $\{\psi = 0\}=\{\phi_\varepsilon = \inf_P \phi_\varepsilon\} \subseteq P^\varepsilon_- = \{\phi_\varepsilon < 0\}$. 
This implies that
$$\int_P |\phi_\varepsilon| = -\int_{P_-^\varepsilon} \phi_\varepsilon + \int_{P_+^\varepsilon} \phi_\varepsilon = -2\int_{P_-^\varepsilon} \phi_\varepsilon \geq -2 \mu(\{\psi = 0\}) \inf_P \phi_\varepsilon  \geq -2(1-\varepsilon)\mu(P) \inf_P \phi_\varepsilon .$$
As a result, the last inequality \eqref{eq: thm_convex} is sharp. Moreover, this argument shows that extremizing functions for this equality do not exist.
\end{remark}

Finally, we discuss the existence of extremizers for the first inequality of \eqref{eq: thm_convex}.  To fix notation, we introduce the set of extremizing potentials:
\begin{equation*}
\mathcal W(P) := \bigg\{\psi: P \to \Bbb R \textup{ convex}, \ \   \int_P \psi =0, \ \textup{ and } \ -\frac{2}{n+1}\cdot{\left(\frac{n}{n+1}\right)}^n\inf_P \psi = \frac{1}{\mu(P)}\int_P |\psi|.\bigg\}
\end{equation*}

Notice that $\phi \in \mathcal W(P)$ implies that $r \phi \in \mathcal W(P)$ for $r > 0$. To deal with this, we will often normalize elements of $\mathcal W(P)$ to satisfy $\inf_P \phi = -1$. Using the language of the proof of Theorem \ref{thm: thm_convex} we notice the following special properties of elements of $\mathcal W(P)$:

\begin{prop}\label{prop:extr_func}Let $B \subset \Bbb R^n$ convex open and bounded and  $\phi \in \mathcal W(P)$ and $\inf_P \phi=\phi(y)=-1$, for some $y \in \overline{P}$. 
Then $\sup_P \phi = \frac{1}{n}$ and 
\begin{equation}\label{eq: sublevelset_identity}
P^\phi_a := \{\phi < a\} = \frac{1-an}{1 + n} y + \frac{na+n}{1+n}P, \ \textup{ for any } \ -1 < a < \frac{1}{n}.
\end{equation}
\end{prop}

\begin{proof}
The argument is a careful analysis of the proof of Theorem \ref{thm: thm_convex} for $\phi \in \mathcal W(P)$ and $\inf_P \phi = -1$. We will also use numerous times that $x \to \mu(P^\phi_x)$ is continuous due to the Brunn-Minkowski inequality ($\phi$ is convex).

Since $\phi \in \mathcal W(P)$ we need to have $A = \frac{1}{n+1}$, otherwise we have strict inequality in \eqref{eq: last_eq}. For similar reasons, we also need that $\mu(P) = \mu (P^\phi_x)$ for $x \geq \frac{1}{n}$, otherwise, we have strict inequality in  \eqref{eq: 1/n_ineq}. This implies that $\sup_P \phi \leq \frac{1}{n}$.

As we have equality in \eqref{eq: homot_ineq_0}, we obtain that $\mu(P^\phi_x) = (1+x)^n \mu(P^\phi_0)$ for $x \in [-1,0]$. Similarly, due to equality in \eqref{eq: homot_ineq} we have that $\mu(P^\phi_x) = (1+x)^n \mu(P^\phi_0)$ for $x \in [0,\frac{1}{n}]$. Putting these last two facts together, we obtain equality in \eqref{eq: P_a_P_b_est}:
$$
\mu\big(P^\phi_b\big) = {\left(\frac{b+1}{a+1}\right)}^n\mu\left(P^\phi_a\right), \  \textup{ for any } -1 < a \leq b < \frac{1}{n}.
$$
Since the measure of the open convex sets on both sides of \eqref{eq: P_a_P_b_est_0} is the same, by Lemma \ref{lem: open_same_measure} below we conclude that $
\frac{b-a}{b+1} y + \frac{a+1}{b+1}P^\phi_b = P^\phi_a$ for any $-1 < a < b< \frac{1}{n}$. Letting $b \nearrow \frac{1}{n}$, we conclude that $
P^\phi_a = \frac{1-an}{1 + n} y + \frac{na+n}{1+n}P$, as desired.

Lastly, we see that $P_a$ is strictly increasing as $a \nearrow \frac{1}{n}$, hence $\sup_P \phi = \frac{1}{n}$.
\end{proof}

\begin{lemma}\label{lem: open_same_measure} Suppose that $A,B \subset \Bbb R^n$, both convex, bounded, open and $B \subset A$. If $\mu(A)=\mu(B)$, then $A=B$.
\end{lemma}
\begin{proof} Let $x \in A$. Since $A$ is open, there exists $y_0,y_1,\ldots,y_n \in A$ such that $x$ is in the interior of the simplex with vertices $y_0,y_1,\ldots,y_n$. Since $\mu(A)=\mu(B) < \infty$, $B$ is dense in $A$. Hence, for all $j \in \{0,\ldots,n\}$ there exists $z_j \in B$ close enough to $y_j$ such that $x$ is in the interior of the simplex with vertices $z_0,z_1,\ldots,z_n$. Since $B$ is convex, it must contain the interior of the simplex with vertices $z_0,z_1,\ldots,z_n \in B$, hence it also contains $x$.
\end{proof}

After carrying out some calculations, we verify that the reverse of Proposition \ref{prop:extr_func} also holds:

\begin{prop}\label{prop: simple_reverse}Let  $\phi \in L^1(P)$ be convex. Suppose there exists $y \in \overline{P}$ such that $P^\phi_a : = \{\phi < a\}= \frac{1-an}{1 + n} y + \frac{na+n}{1+n}P$ for any $-1 < a < \frac{1}{n}.$ Then $\int_P \phi = 0$ and $\phi \in \mathcal W(P)$.
\end{prop} 

\begin{proof} Letting $a \searrow -1$ and $a \nearrow \frac{1}{n}$  in $P^\phi_a : = \{\phi < a\}= \frac{1-an}{1 + n} y + \frac{na+n}{1+n}P$ we conclude that $\inf_P \phi=\phi(y)=-1$ and $\sup_P \phi = \frac{1}{n}$. Moreover, after taking measures, we also have that $\mu(P^\phi_x) = (1+x)^n \mu(P^\phi_0)$ for $x \in [-1, \frac{1}{n}]$.

Next we verify that $\int_P \phi = 0$. For this we carry out the following side calculations

\begin{align*}
  \int_{P^\phi_-} |\phi| = \int_{-1}^0 \mu\left(P^\phi_x\right)dx = \int_{-1}^0 {(1+x)}^n\mu\big(P^\phi_0\big)dx = \frac{1}{n+1}\mu(P^\phi_0),
\end{align*}

\begin{align*} 
    \int_{P^\phi_+} \phi &= \int_0^\frac{1}{n} \mu\left(P \setminus P^\phi_x\right)dx =\int_0^\frac{1}{n} \mu\left(P\right) - \mu\left(P^\phi_x\right)dx= \int_0^\frac{1}{n} \mu\left(P\right) - {(1+x)}^n\mu\big(P^\phi_0\big)dx \\
 &= \frac{1}{n}\mu(P) - \frac{{(\frac{1}{n}+1)^{n+1}}}{n+1}\mu(P^\phi_0) + \frac{1}{n+1}\mu(P^\phi_0)\\
 &=\frac{1}{n}\bigg(1+\frac{1}{n}\bigg)^n\mu(P^\phi_0) - \frac{{(\frac{1}{n}+1)^{n+1}}}{n+1}\mu(P^\phi_0) + \frac{1}{n+1}\mu(P^\phi_0) = \frac{1}{n+1}\mu(P^\phi_0). 
\end{align*}
Comparing the above, we conclude that $\int_P \phi =\int_{P_+} \phi+\int_{P_-} \phi= 0$, as desired.

Finally, using the above facts, we see that each inequality in the proof of Theorem \ref{thm: thm_convex} is in fact an equality, ultimately yielding:
$$
 -\frac{2}{n+1}\cdot{\left(\frac{n}{n+1}\right)}^n\inf_P \phi = \frac{1}{\mu(P)}\int_P |\phi| d\mu.
$$
\end{proof}

With the previous two propositions in hand, we obtain the following characterization of extremizers:

\begin{theorem}\label{thm: extremizer_char} For any $\phi \in \mathcal W(P)$ there exists a unique $y_\phi \in \overline{P}$ such that $\inf_P \phi = \phi(y_\phi)$. Moreover, the map $F: \mathcal W(P) \cap \{\inf_P \phi = -1\} \to \overline{P}$ given by $F(\phi)= y_\phi$ is a bijective function. 
\end{theorem}

As pointed out to us by the anonymous referee, it is an intriguing question to give a characterization of elements in $\mathcal W(P)$ in terms of their Legendre transform.

\begin{proof} Let $\phi \in \mathcal W(P)$. We can assume without loss of generality that $\inf_P \phi =-1$.
That $y_\phi \in \overline{P}$ is unique, follows from \eqref{eq: sublevelset_identity}. This proves that the map $F(\phi) = y_\phi$ is well defined.

 To show that $F$ is injective on $\mathcal W(P) \cap \{\inf_P \phi = -1\}$, we notice that the identity $P^\phi_a := \{\phi < a\}= \frac{1-an}{1 + n} y_\phi + \frac{na+n}{1+n}P, \ a \in (-1, \frac{1}{n})$ determines $\phi$ uniquely.

To argue surjectivity, let $y \in \overline{P}$. We are going to construct an extremizer $\phi_y \in \mathcal W(P)$ satisfying $\inf_P \phi_y= -1$ and $F(\phi_y) = y$.

Let $x \in P$. We define $\phi_y(x)$ as follows:

$$\phi_y(x) = \inf\Big\{r  \in \Big[-1,\frac{1}{n}\Big] \textup{ such that } x \in \frac{1-rn}{1 + n} y + \frac{nr+n}{1+n}P\Big\}$$

By definition of $\phi_y$ we have:
\begin{equation}\label{eq: sublevel_set}
P^\phi_a: = \{\phi_y < a \} = \frac{1-an}{1 + n} y + \frac{na+n}{1+n}P, \ a \in \Big[-1, \frac{1}{n}\Big].
\end{equation}

We argue that $\phi_y$ is convex. Let $x,z \in P$ and suppose that $\phi_y(x) < a$ and $\phi_y(z) < b$. We will argue that $\phi_y(\alpha x + (1-\alpha) y) \leq \alpha a + (1-\alpha)b, \ \alpha \in [0,1]$. This will follow if we can argue that $\alpha \{\phi_y < a\} + (1-\alpha)\{\phi_y < a\} \subseteq \{\phi_y < \alpha a + (1-\alpha)b\}$. By \eqref{eq: sublevel_set} this is equivalent with:

$$\alpha \frac{na+n}{1+n}P + (1-\alpha) \frac{nb+n}{1+n}P \subseteq \frac{n(\alpha a + (1-\alpha)b)+n}{1+n}P.$$
But this inclusion follows from the fact that $P$ is convex.

Lastly, \eqref{eq: sublevel_set} and Proposition \ref{prop: simple_reverse} implies that $\int_P \phi_y =0$, as well as $\phi_y \in \mathcal W(P)$.
\end{proof}

\subsection{The case of affine extremizers}

Given an extremizer $\phi \in \mathcal W(P)$, whenever $y_\phi \in P$ in Theorem \ref{thm: thm_convex}, the graph of $\phi$ is a cone, with vertex at $y_\phi$. However when $y_\phi \in \partial P$, the shape of the graph of $\phi$ may change drastically. In particular, as pointed out in the proof of \cite[Proposition 4.2]{DGS21}, in some instances such extremizers $\phi$ are affine functions. In this short subsection, we point out that this can happen only in a very particular instance: when $P$ is the intersection of a convex cone and a  half-space.

Let $L \subset \Bbb R^n$ open. We say that $L$ is a \emph{linear cone} if $L+ L \subset L$ and $\alpha L \subset L$ for any $\alpha >0$.
We say that $C \subset \Bbb R^n$ is a \emph{convex cone} if $C = y + L$ for some $y \in \Bbb R^n$ and $L$ linear cone.

We say that $H \subset \Bbb R^n$ is an open half-space, if there exist $\chi: \Bbb R^n \to \Bbb R$ linear so that $H = \{\chi < c\}$ for some $c\in \Bbb R$. If $0 \in H$, we have that $c > 0$, and in this case we can always arrange that $H = \{\chi < 1 + \frac{1}{n}\}$. This normalization determines $\chi: = \chi_H$ uniquely, and we will call $\chi_H$ the \emph{linear defining function} of $H$.

\begin{theorem} \label{thm: convex_extr_aff}Let $P \subset \Bbb R^n$ be a bounded open convex set. There exists $\phi \in \mathcal W(P)$ affine if and only if $P = C \cap H$, where $C=y +L$ is a convex cone and $H$ is a  half-space. In the latter case $\psi(x) := \chi_{H-y}(x-y) -1 \in \mathcal W(P)$ is an affine extremizer.
\end{theorem}

\begin{proof}Suppose that $\phi: P \to \Bbb R$ affine and $\phi \in \mathcal W(P)$. Let $y \in \overline{P}$ such that $\phi(y) = \inf_P \phi$. 
After replacing $P$ with $P -y$ and $\phi(x)$ with $\phi(x+y)$ we can assume that $y=0$. We can also assume that $\inf_P \phi = -1$.

In particular, $\psi(x):= \phi(x)+1, 
 x \in \Bbb R^n$ is linear, since $\psi$ is affine and $\psi(0)=0$. We introduce $H := \{x \in \Bbb R: \phi(x) < \frac{1}{n}\}=\{x \in \Bbb R : \psi(x) < 1+\frac{1}{n}\}$.
 
 Let $L := \{r P, \ r > 0\}$. It is clear that $L$ is a linear cone, since $P$ is convex. We claim that $P = L \cap H$. Since $P =P^\phi_{\frac{1}{n}} \subset H$, the easy inclusion is $P \subset L \cap H$. For the reverse inclusion, let  $p \in L \cap H$. Then there exists $q \in P$ and $r >0$ such that $p = r q$ and $p \in H$. The latter implies that 
\begin{equation}\label{eq: interm_ineq}
\psi(p) = \phi(p) + 1 =  r\psi(q) = r(\phi(q) + 1) < 1 + \frac{1}{n}.
\end{equation} 
 
If $r \in [0,1]$, then we are done, since $p=r q = 0(1-r) + rq \in P$ ($0 \in \overline{P}$ and $P$ is open). If, $r > 1$ then by \eqref{eq: interm_ineq} we have $\phi(p/r) < \frac{1}{r}-1 + \frac{1}{nr}.$ This implies that $p/r \in P_{\frac{1}{r}-1 + \frac{1}{nr}}$. Since $\phi \in \mathcal W(P)$, by \eqref{eq: sublevelset_identity} we get that 
$$p \in r P_{\frac{1}{r}-1 + \frac{1}{nr}} = P.$$

For the reverse direction, assume that if $P = C \cap H$, where $C = y + L$ is a convex cone and $H$ is half-space. Let $\chi_{H-y}$ be the  linear defining function of $H$.

In this case we take $\phi(x) := \chi_{H-y}(x-y) -1$. We notice that $\phi(y) = -1$, $\sup_P \phi = \frac{1}{n}$. Moreover, since $\phi$ is affine, we obtain that 
$$P^\phi_a := \{\phi < a\}= \frac{1-an}{1 + n} y + \frac{na+n}{1+n}P, \ \  -1 < a < \frac{1}{n}.$$
Proposition \ref{prop: simple_reverse} now gives and $\int_P \phi = 0$ and $\phi \in \mathcal W(P)$.
\end{proof}

\section{Extremizing rays and plurisupported currents}

In proving Theorem \ref{thm: main_extremal_ray} we heavily rely on the formalism developed in \cite{DL18} regarding the metric space of geodesic rays. For more background and details we refer to this work.

By $\mathcal E^1 \subset \textup{PSH}(X,\theta)$ we denote the space of finite energy pontentials: $u \in \mathcal E^1$ if $\int_X \theta_u^n = \int_X\omega^n$ (where $\theta_u^n$ is the non-pluripolar complex Monge--Ampere measure, defined in \cite[Section 1]{GZ07}), moreover $\int_X |u| \theta_u^n < \infty$.

By $\mathcal R^1$ we denote the space of $L^1$ Mabuchi geodesic rays $[0,\infty) \ni t \to u_t \in \mathcal E^1$, normalized by $u_0 =0$ and $I(u_t) =0, \ t \geq 0$.

By \cite[Proposition 5.1]{BDL1} the map $t \to d_1(u_t,v_t)$ is convex for any  $\{u_t\}_t,\{v_t\}_t \in \mathcal R^1$, where we used the conventions of Remark \ref{rem: rays_pot}. In \cite{DL18} this was used to define the following metric (see also \cite{CC18}):
$$d_1^c(\{u_t\}_t,\{v_t\}_t) = \lim_{t \to \infty} \frac{d_1(u_t,v_t)}{t}.$$
By \cite[Theorem 1.3 and 1.4]{DL18} we know that $(\mathcal R^1,d_1^c)$ is complete, moreover the space of normalized bounded geodesic rays $\mathcal R^\infty$ is dense in $\mathcal R^1$ \cite[Theorem 1.5]{DL18}.

The radial $J$ functional from \eqref{eq: J_radial_def} can be expressed in very simple terms, slightly extending \cite[Lemma 5.2]{DGS21}:

\begin{lemma}\label{lem: J_rad_formula} For $\{u_t\}_t \in \mathcal R^1$ we have that $J\{u_t\} = \sup_X \dot u_0= \tau^+_u=\frac{\sup_X u_l}{l}$ for any $l >0$, where $\dot u_0 := \lim_{t \to 0} \frac{u_t}{t}$.
\end{lemma}

\begin{proof} We have that $J(u_t) = \frac{1}{V} \int_X u_t \omega^n - I(u_t) = \frac{1}{V} \int_X u_t \omega^n$. By \cite[Lemma 3.45]{Da19} we obtain that 
$$J\{u_t\} = \lim_{t \to \infty} \frac{J(u_t)}{t} = \lim_{t \to \infty} \frac{\sup_X u_t}{t}.$$
By Lemma \ref{lem: suplinear}  $t \to \sup_X u_t$ is linear, hence $J\{u_t\} = \frac{\sup_X u_l}{l}=\tau^+_u$ for any $l>0$.

Finally, we argue that $J\{u_t\} = \sup_X \dot u_0$. For this we have to argue that $\sup_X \dot u_0 = \frac{\sup_X u_l}{l}$ for any $l>0$.  We make use of Theorem \ref{thm: max_test_curve_ray_duality}. To start, we notice that we can assume that $0=\tau^+_u = \sup_X \frac{u_l}{l}, \ l >0$, after replacing $u_t$ with $u_t - \tau^+_u t$.

Since $u_t : = \sup_{\tau \leq 0}(\hat u_\tau + t\tau), \ t \geq 0$, we obtain that $u_t \geq \hat u_0$ for any $t \geq 0$. From $\tau^+_u = 0$ we also obtain that $u_t \leq 0$. Since $\{\hat u_\tau\}_\tau$ is maximal, we obtain that $P[\hat u_0] = \hat u_0$, in particular $\sup_X \hat u_0 = 0$. Moreover, since $\hat u_0$ is usc, there exists $x \in X$ such that $\sup_X \hat u_0 = u_0(x)$. Since $\hat u_0 \leq u_t \leq 0$, we obtain that $u_t(x) =0$, for all $t >0$.
By $t$-convexity, we have $\dot u_0 \leq 0$. But since $\dot u_0(x)=0$, we conclude that $\sup_X \dot u_0 = 0 = \frac{\sup_X u_l}{l}$ for any $l>0$, finishing the proof.
\end{proof}

\begin{lemma}\label{lem: d_1_init} For $\{u_t\}_t \in \mathcal R^1$ we have $d_1(0,u_1) = \int_X |\dot u_0|\omega^n$ and $0=I(u_1) = \int_X \dot u_0 \omega^n$.
\end{lemma}

\begin{proof} That $d_1(0,u_1) = \int_X |\dot u_0|\omega^n$  follows from \cite[Lemma 3.4]{BDL2}. What is more, the argument of \cite[Lemma 3.4]{BDL2} is seen to imply $I(u_1) = \int_X \dot u_0 \omega^n$.
\end{proof}

Given $v \in \textup{PSH}(X,\theta)$, we say that $v$ is a \emph{model potential} if $v = P[v]$, where 
$P[v]$ was defined in \eqref{eq: P[v]_def}.
For more on model potentials we refer to \cite{DDL2}.

\begin{lemma}\label{lem: rad_legendre_sublevel} For any $\{u_t\}_t \in \mathcal R^1$ and $b \leq a <  \tau^+_u$ we have that
$$ \int_{\{\dot u_0 \geq b\}} \omega^n \leq \frac{(\tau^+_u - b)^n}{(\tau^+_u - a)^n} \int_{\{\dot u_0 \geq a\}} \omega^n.$$
Moreover $[-\infty, \tau^+_u) \ni s \to \big(\int_{\{\dot u_0 \geq s\}} \omega^n\big)^{1/n}$ is concave.
\end{lemma}

\begin{proof} It follows from Theorem \ref{thm: max_test_curve_ray_duality} that $\hat u_\tau \in \textup{PSH}(X,\omega)$ is a model potential. As a result, by  \cite[Theorem 1]{DT20} (c.f. \cite[Theorem 3.8]{DDL2}) we have that 
$\int_X \omega_{\hat u_\tau}^n = \int_{\{ \hat u_\tau = 0\}} \omega^n$. 
Moreover, due to basic properties of Legendre transforms $\{\dot u_0 \geq \tau \} = \{ \hat u_{\tau} = 0\}$, in particular,
\begin{equation}\label{eq: int_est}
\int_X \omega_{\hat u_{\tau}}^n = \int_{\{ \dot u_0 \geq \tau\}} \omega^n.
\end{equation}
Let now $b \leq a \leq \tau^+_u$. We recall that $\tau \to \hat u_\tau$ is concave and  $\textup{PSH}(X,\omega) \ni v \to \big(\int_X \omega_{ v}^n\big)^{\frac{1}{n}}$ is concave as well (\cite[Theorem B]{DDL4}. As a result, using \cite[Theorem 1.2]{WN17} we conclude that $s \to \big(\int_X \omega_{\hat u_s}^n \big)^{1/n}$ is concave. Together with \eqref{eq: int_est} this impies the last statement of the lemma. Using concavity we can write that
\begin{flalign}\label{eq: ineq 0}
\frac{(\tau^+_u - a)}{(\tau^+_u - b)}  \bigg(\int_X \omega_{\hat u_b}^n\bigg)^{\frac{1}{n}} &\leq 
\frac{(\tau^+_u - a)}{(\tau^+_u - b)}  \bigg(\int_X \omega_{\hat u_b}^n\bigg)^{\frac{1}{n}} +\frac{(a - b)}{(\tau^+_u - b)}  \bigg(\int_X \omega_{\hat u_{\tau^+_u}}^n\bigg)^{\frac{1}{n}} \\
\label{eq: ineq 1}&\leq \bigg(\int_X \omega_{\frac{(\tau^+_u - a)}{(\tau^+_u - b)} \hat u_b + \frac{(a - b)}{(\tau^+_u - b)} \hat u_{\tau^+_u}}^n\bigg)^{\frac{1}{n}}\\
&\leq \bigg(\int_X \omega_{\hat u_a}^n\bigg)^{\frac{1}{n}}. \label{eq: ineq 2}
\end{flalign}
Comparing with \eqref{eq: int_est}, the result follows.
\end{proof}

Before we proceed, let us recall the construction for a special type of geodesic ray from \cite{Da17}, associated to any $v \in \textup{PSH}(X,\omega)$ and $a,b \in \Bbb R$ with $a<b$.

With \eqref{eq: Leg_transf_def_ray_test_curve} and Theorem \ref{thm: max_test_curve_ray_duality} in mind, let $\hat r(a,b,v)_\tau \in \textup{PSH}(X,\omega)$ defined as follows:
\begin{equation}
\label{eq:psi=P[u]}
 \hat r(a,b,v)_\tau:=
    \begin{cases}
   0,\ &\tau\leq a;\\
   P[\frac{\tau -a}{b-a} v],\ &a < \tau <b;\\
   \lim_{\tau\nearrow b}P[\frac{\tau -a}{b-a} v],\ &\tau=b;\\
   -\infty,\ &\tau>b.
\end{cases}
\end{equation}
By \cite[Theorem 3.12]{DDL2} we have that $P[r(a,b,v)_\tau]=r(a,b,v)_\tau, \ \tau \in \Bbb R$. Hence, by Theorem \ref{thm: max_test_curve_ray_duality}(iii) we obtain that $\{\hat r(a,b,v)_\tau\}_\tau$ is a bounded maximal test curve, inducing a geodesic ray $\{r(a,b,v)_t\}_t \in \mathcal R^\infty$ that will play a critical role below, in case $v \in \textup{PSH}(X,\omega)$ is plurisupported. Indeed, for $v$ plurisupported it is possible to compute the radial Monge--Amp\`ere and $J$ energies of $\{r(a,b,v)_t\}_t$:

\begin{prop}\label{prop: v_J_d_1_extremal}For any $v \in \textup{PSH}(X,\omega)$ plurisupported we have  $I\{r(-{1}/{n},1,v)_t\} =0$, moreover
$$\frac{2}{n+1}\cdot{\left(\frac{n}{n+1}\right)}^n J\{r(-{1}/{n},1,v)_t\} = d_1(0,r(-{1}/{n},1,v)_1).$$
\end{prop}
\begin{proof} We start with computing $I\{r(-{1}/{n},1,v)_t\}$. Notice that $\tau_{r}^+ = 1$ and $\tau_{r}^- = -1/n$. 

We will use numerous times that $\int_X \omega_{P[w]}^n = \int_X \omega_{w}^n$ for all $w \in \textup{PSH}(X,\omega)$ (\cite[Theorem 1.3, Theorem 2.3]{DDL2}). In addition, since $v$ is plurisupported, we have $\int_X \omega^k_v \wedge \omega^{n-k}=0$ for all $k \in \{1,\ldots,n\}$. 

Using the above, formula \eqref{eq: I_RWN_form} allows to carry out the  following calculations:
\begin{flalign}\label{eq: I_formula}
            I\{r(-{1}/{n},1,v)_t\} &=\frac{1}{V}\int_{-\frac{1}{n}}^{1} \left(\int_X \omega_{P[\frac{n\tau + 1}{n+1} v]}^n-\int_X \omega^n \right) \,\mathrm{d}\tau+ 1 \\
            &= \frac{1}{V}\int_{-\frac{1}{n}}^{1} \left(\int_X \omega_{\frac{n\tau + 1}{n+1} v}^n-\int_X \omega^n \right) \,\mathrm{d}\tau+ 1 \nonumber\\
            &= \frac{1}{V}\int_{-\frac{1}{n}}^{1} \left( \bigg(\frac{n - n\tau}{n+1}\bigg)^n \int_X \omega^n-\int_X \omega^n \right) \,\mathrm{d}\tau+ 1 \nonumber\\
            & = \frac{n^n}{(n+1)^n} \int_{-\frac{1}{n}}^{1}(1-\tau)^nd\tau - \frac{1}{n}=0. \nonumber
        \end{flalign}
        
Since $J\{r(-{1}/{n},1,v)_t\} = \tau^+_{r(-{1}/{n},1,v)} - I\{r(-{1}/{n},1,v)_t\}$ (Lemma \ref{lem: J_rad_formula}), and $\tau^+_{r(-{1}/{n},1,v)}  =1$, we obtain that $J\{r(-{1}/{n},1,v)_t\} = 1$.

Lastly, we compute $d_1(0,r(-{1}/{n},1,v)_1) = \frac{1}{V}\int_X  |\dot r(-{1}/{n},1,v)_0| \omega^n$. 
 \begin{flalign*}
            d_1(0,r(-{1}/{n},1,v)_1) &=\frac{1}{V}\int_X  |\dot r(-{1}/{n},1,v)_0| \omega^n \\
            &=\frac{2}{V}\int_{_{\{\dot r(-{1}/{n},1,v)_0\geq 0 \}}}  \dot r(-{1}/{n},1,v)_0 \omega^n\\
            &= \frac{2}{V} \int_0^1 \omega^n(\{\dot r(-{1}/{n},1,v)_0\geq s\})ds\\
            &= \frac{2}{V} \int_0^1 \omega_{\hat r_s}^n ds = \frac{2}{V} \int_0^1 \omega_{\frac{ns + 1}{n+1}v}^n ds\\            
            &= \frac{2}{V}\int_{0}^{1}  \bigg(\frac{n - ns}{n+1}\bigg)^n \int_X \omega^n \,\mathrm{d}s \\
                        &= 2\int_{0}^{1}  \bigg(\frac{n - ns}{n+1}\bigg)^n \,\mathrm{d}\tau = \frac{2n^n}{(n+1)^{n+1}},
        \end{flalign*}
where in the second line we have used Lemma \ref{lem: d_1_init}, in fourth line we have used \eqref{eq: int_est}, and in the fifth line we have used that $v$ is plurisupported.
\end{proof}

With the above in hand, we are ready characterize the case of equality in Proposition \ref{lem: rad_legendre_sublevel}:

\begin{lemma}\label{lem: rad_legendre_sublevel_eq} Let $\{u_t\}_t \in \mathcal R^\infty$. Then
$$ \int_{\{\dot u_0 \geq b\}} \omega^n = \frac{(\tau^+_u - b)^n}{(\tau^+_u - a)^n} \int_{\{\dot u_0 \geq a\}} \omega^n$$
for all  $\tau^-_u \leq b \leq a <  \tau^+_u$ if and only if there exists $v \in \textup{PSH}(X,\omega)$ plurisupported such that $u_t = r(\tau^-_u,\tau^+_u,v))_t, \ t \geq 0$.
\end{lemma}

\begin{proof} We argue the forward direction only, as this is the only implication that we will use. We leave the simple proof of the reverse direction to the reader.

Let $v = \hat u_{\tau^+_u}$. We argue that $v$ is plurisupported and  $u_t = r(\tau^-_u,\tau^+_u,v))_t, \ t \geq 0$.

To start we note that all the inequalities \eqref{eq: ineq 0},\eqref{eq: ineq 1} and \eqref{eq: ineq 2} are equalities.  Taking $b = \tau^-_u$ and $a \in (\tau^-_u,\tau^+_u)$ we arrive at:
\begin{flalign}\label{eq: ineq 0p}
\frac{(\tau^+_u - a)}{(\tau^+_u - \tau^-_u)}  \bigg(\int_X \omega^n\bigg)^{\frac{1}{n}}&= 
\frac{(\tau^+_u - a)}{(\tau^+_u - \tau^-_u)}  \bigg(\int_X \omega^n\bigg)^{\frac{1}{n}} +\frac{(a - {\tau^-_u})}{(\tau^+_u - {\tau^-_u})}  \bigg(\int_X \omega_{v}^n\bigg)^{\frac{1}{n}} \\
\label{eq: ineq 1p}&= \bigg(\int_X \omega_{\frac{(a - {\tau^-_u})}{(\tau^+_u - {\tau^-_u})} v}^n\bigg)^{\frac{1}{n}}\\
&= \bigg(\int_X \omega_{\hat u_a}^n\bigg)^{\frac{1}{n}}. \label{eq: ineq 2p}
\end{flalign}

The equality \eqref{eq: ineq 2p} and  the fact $\hat u_a$ is model (Theorem \ref{thm: max_test_curve_ray_duality}) implies that $\hat u_a = P\Big[\frac{(a - {\tau^-_u})}{(\tau^+_u - {\tau^-_u})} v\Big]=\hat r(\tau^-_u,\tau^+_u,v))_a$ for  $a \in (\tau^-_u,\tau^+_u)$. 

From \eqref{eq: ineq 1p} we have that 
$$\frac{(\tau^+_u - a)^n}{(\tau^+_u - \tau^-_u)^n}  \int_X \omega^n =\int_X \omega_{\frac{(a - {\tau^-_u})}{(\tau^+_u - {\tau^-_u})} v}^n, \ \ a \in [\tau^-_u,\tau^+_u].$$ 
The multilinearity of nonpluripolar products now implies that $\int_X \omega_v^k \wedge \omega^{n-k} =0$ for all $k \in \{1,\ldots,n\}$, hence $v$ is plurisupported.
\end{proof}

Before we characterize extremizing rays, let us recall the proof of  Theorem \ref{thm: radial_main_ineq}  from \cite{DGS21}, sharing striking similarities with the argument of Theorem \ref{thm: thm_convex}:

\begin{theorem}\cite{DGS21}\label{thm: radial_ineq} Suppose that $(X,\omega)$ is a compact K\"ahler manifold and $\{u_t\}_t \in \mathcal R^1$. Then the following  inequalities hold:
\begin{equation}\label{eq: radial_opt_est}
\frac{2}{n+1}\cdot{\left(\frac{n}{n+1}\right)}^n J\{u_t\} \leq d_1(0,u_1)\leq 2 J\{u_t\}.
\end{equation}
\end{theorem}

\begin{proof} Due to Lemma \ref{lem: J_rad_formula} and Lemma \ref{lem: d_1_init}, it is enough to argue the following estimates:
$$\frac{2}{n+1}\cdot{\left(\frac{n}{n+1}\right)}^n \sup_X \dot u_0 \leq \frac{1}{V} \int_X |\dot u_0| \omega^n \leq 2 \sup_X \dot u_0.
$$
Using re-scaling in time, we can further assume that $\sup_X \dot u_0 = 1$, hence it is enough to argue that
\begin{equation}\label{eq: radial_main_dot_ineq}
\frac{2}{n+1}\cdot{\left(\frac{n}{n+1}\right)}^n  \leq \frac{1}{V} \int_X |\dot u_0| \omega^n \leq 2.
\end{equation}
We first argue the easier second estimate.
Let $X_- := \{\dot u_0 < 0\}$ and $X_+ := \{\dot u_0 \geq  0\}$. Since $I(u_1) = \int_X \dot u_0 \omega^n$ we have that
\begin{flalign*}
\frac{1}{V}\int_X |\dot u_0| \omega^n=\frac{2}{V}\int_{X_+} \dot u_0 \omega^n \leq 2 \sup_X \dot u_0.
\end{flalign*}
To address the first estimate, we make the following preliminary calculation:
\begin{flalign}\label{eq: interm_est}
    \int_X |\dot u_0| \omega^n &= 2 \int_{X^+} \dot u_0 \omega^n = 2\int_{0}^1 \int_{\{ \dot u_0 \geq x\}} \omega^ndx \geq 2\int_{0}^1 {(1-x)}^n  \int_{\{ \dot u_0 \geq 0\}}\omega^n dx  \nonumber\\
    &= \frac{2}{n+1}\int_{\{ \dot u_0 \geq 0\}}\omega^n,
\end{flalign}
where we used that $\int f d \mu = \int_0^{+\infty} \mu\{f \geq t\} dt$ for any non-negative $\mu$-measurable $f$, and Lemma \ref{lem: rad_legendre_sublevel} for the parameters $0 \leq x \leq \sup_X \dot u_0 = 1$.

To estimate $\int_{X_-}|\dot u_0| \omega^n$, we can use a similar technique to the above:
\begin{align}\label{eq:1/n_est}
    \int_{X_-}|\dot u_0| \omega^n &\geq \int_0^{\frac{1}{n}} \int_{X \setminus \{\dot u_0 \geq -x \}} \omega^n dx =\int_0^\frac{1}{n} \bigg(V - \int_{ \{\dot u_0 \geq -x \}} \omega^n \bigg)dx\\
    &\geq \int_0^\frac{1}{n} \bigg( V  - {(1+x)}^n\int_{ \{\dot u_0 \geq 0 \}} \omega^n \bigg)dx \nonumber \\
    &= \frac{V}{n} - \frac{{(\frac{1}{n}+1)^{n+1}}}{n+1}\int_{ \{\dot u_0 \geq 0 \}} \omega^n + \frac{1}{n+1}\int_{ \{\dot u_0 \geq 0 \}} \omega^n, \nonumber
\end{align}
where in the second line we used Lemma \ref{lem: rad_legendre_sublevel} again, for the parameters $-x \leq 0 \leq \sup_X \dot u_0 = 1$.

We now let $A>0$ such that $\frac{1}{2}\int_{X}|\dot u_0| \omega^n =\int_{X_+} |\dot u_0|  \omega^n= -\int_{X_-} \dot u_0  \omega^n = A\int_{ \{\dot u_0 \geq 0 \}} \omega^n$.  This gives
\[A \int_{ \{\dot u_0 \geq 0 \}} \omega^n\geq \frac{1}{n}V - \frac{1}{n}\cdot{\left(\frac{n+1}{n}\right)}^{n}\int_{ \{\dot u_0 \geq 0 \}} \omega^n + \frac{1}{n+1}\int_{ \{\dot u_0 \geq 0 \}} \omega^n,\]
implying
\[\int_X |\dot u_0| \omega^n=2A\int_{ \{\dot u_0 \geq 0 \}} \omega^n \geq \frac{2A}{nA + {\left(\frac{n+1}{n}\right)}^n - \frac{n}{n+1}}V.\]
The right-hand side is an increasing function of $A$ and by \eqref{eq: interm_est} we know $A \geq \frac{1}{n+1}$.  This means the right hand side is minimized at $A = \frac{1}{n+1}$, hence
\begin{equation}\label{eq: last_ineq}
\int_X |\dot u_0| \omega^n \geq \frac{\frac{2}{n+1}}{{\left(\frac{n+1}{n}\right)}^n}V
= \frac{2}{(n+1)}\cdot{\left(\frac{n}{n+1}\right)}^nV,
\end{equation}
finishing the proof of \eqref{eq: radial_main_dot_ineq}.
\end{proof}

\begin{theorem}\label{thm: extr_ray_char}Suppose that $(X,\omega)$ is a compact K\"ahler manifold and $\{u_t\}_t \in \mathcal R^1$ is such that $\tau^+_u=1$ and
\begin{equation}\label{eq: radial_opt_eq}
\frac{2}{n+1}\cdot{\left(\frac{n}{n+1}\right)}^n J\{u_t\} = d_1(0,u_1).
\end{equation}Then there exists $v \in \textup{PSH}(X,\omega)$ plurisupported such that $u_t = r(-\frac{1}{n},1,v)_t, \ t \geq 0$. Conversely, for any $w \in \textup{PSH}(X,\omega)$ plurisupported, the ray $\{r(-\frac{1}{n},1,w)_t\}_t$ satisfies \eqref{eq: radial_opt_eq}.
\end{theorem}

We note that the condition $\tau^+_u=1$ is a simple normalization, and can be always be achieved after rescaling in time any ray $\{u_t\}_t \in \mathcal R^1$. \eqref{eq: radial_opt_eq}. Though toric symmetries are not involved, the argument shares similarities with the proof of Proposition \ref{prop:extr_func}.

\begin{proof} Let $\{u_t\}_t \in \mathcal R^1$ such that \eqref{eq: radial_opt_eq} holds.

Examining the proof of Theorem \ref{thm: radial_ineq}, since all $x$-integrands are continuous (Lemma \ref{lem: rad_legendre_sublevel}), we conclude that all the inequalities between \eqref{eq: interm_est} and \eqref{eq: last_ineq} are in fact equalities for our $\{u_t\}_t \in \mathcal R^1$.
In particular the first line of \eqref{eq:1/n_est} implies that $\dot u_0 \geq- \frac{1}{n}$ a.e. on $X$. This gives $u_t \geq -\frac{t}{n}$ a.e. on $X$, by $t$-convexity. Since $u_t \in \textup{PSH}(X,\omega)$ we obtain that $u_t \geq -\frac{t}{n}$ globally, hence  $\{u_t\}_t \in \mathcal R^\infty$.

Equality in the second line of \eqref{eq:1/n_est} and \eqref{eq: interm_est} gives $\int_{\{\dot u_0 \geq x\}} \omega^n = (1-x)^n \int_{\{ \dot u_0 \geq 0\}}\omega^n$ for any $x \in [-\frac{1}{n},1)$. Picking  $x = a$, and $x=b$ for $-\frac{1}{n}=\tau^-_u \leq b \leq a <  \tau^+_u=1$  we arrive at
$$ \int_{\{\dot u_0 \geq b\}} \omega^n = \frac{(1 - b)^n}{(1 - a)^n} \int_{\{\dot u_0 \geq a\}} \omega^n.$$
Lemma \ref{lem: rad_legendre_sublevel_eq} now implies that  there exists $v \in \textup{PSH}(X,\omega)$ plurisupported such that $u_t = r(\tau^-_u,\tau^+_u,v))_t, \ t \geq 0$.

For the reverse direction, let $u_t := r(-1/n,1,w)_t, \ t \geq 0$ for $w \in \textup{PSH}(X,\omega)$ plurisupported. Then Proposition \ref{prop: v_J_d_1_extremal} implies that \eqref{eq: radial_opt_eq} holds.
\end{proof}

\let\omegaLDthebibliography\thebibliography 
\renewcommand\thebibliography[1]{
  \omegaLDthebibliography{#1}
  \setlength{\parskip}{1pt}
  \setlength{\itemsep}{1pt plus 0.3ex}
}

\bigskip
\normalsize
\noindent{\sc University of Texas, Austin}\\
{\tt abenda@utexas.edu}\vspace{0.1in}\\
\noindent{\sc University of Maryland}\\
{\tt sbachhub@terpmail.umd.edu, bchristo@terpmail.umd.edu, tdarvas@umd.edu}
\end{document}